\documentclass[a4j,12pt, leqno]{article}

\usepackage{amsmath,amsthm,amssymb}
\usepackage[dvipdfmx]{color}
\usepackage[dvipdfmx]{graphicx}
\usepackage{latexsym}
\usepackage{ascmac}
\usepackage{longtable}
\usepackage{color}

\newcommand{\mb}{\mathbb}
\newcommand{\R}{\mathbb{R}}
\newcommand{\dom}{\mathrm{dom}}

\theoremstyle{plain}
\newtheorem{thm}{Theorem}
\newtheorem{lem}{Lemma}
\newtheorem{cor}{Corollary}
\newtheorem{prop}{Proposition}
\newtheorem{asm}{Assumption}

\newtheorem{exm}{Example}
\newtheorem{defn}{Definition}

\theoremstyle{remark}

\begin{document}

\title{Duality of Optimization Problems \\with Gauge Functions%
\thanks{
This work 
was supported in part by a Grant-in-Aid for Scientific Research (C) (17K00032) from 
Japan Society for the Promotion of Science.
}}

\author{
Shota Yamanaka%
\thanks{Mitsubishi Chemical Corporation, Okayama 712-8054, Japan. Email: \texttt{yamanaka.shota.52w@kyoto-u.jp}}
\ and
Nobuo Yamashita
\thanks{Department of Applied Mathematics and Physics, 
Graduate School of Informatics, Kyoto University,
Kyoto 606-8501, Japan. Email: \texttt{nobuo@i.kyoto-u.ac.jp}}
}

\date{December 14, 2017\\
Revised: December 25, 2020}

\maketitle

\begin{abstract}
\noindent
Recently, Yamanaka and Yamashita proposed the so-called 
positively homogeneous optimization problem, which includes many important
problems, such as the ab\-so\-lute-value and the gauge optimization problems.
They presented a closed form of the dual formulation for the problem, and showed weak duality
and the equivalence to the Lagrangian dual under some conditions.
In this work, we focus on a special positively homogeneous optimization problem, whose objective function and constraints
consist of some gauge and linear functions.
We prove not only weak duality but also strong duality. 
We also study necessary and sufficient optimality conditions associated to the problem.
Moreover, we give sufficient conditions under which
we can recover a primal solution from a Karush-Kuhn-Tucker point of the dual formulation.
Finally, we discuss how to extend the above results to general convex optimization problems by considering the so-called
perspective functions.\\

\end{abstract}


\section{Introduction}

The \emph{gauge optimization} (GO) problem is described as follows~\cite{ABD2017,F1987,FM2016,FMP2014}:
\begin{equation}\label{eq:go}
\min_{x \in \mathcal{X}} \quad g(x),
\tag{\mbox{$\mathrm{P_{GO}}$}}
\end{equation}
where $\mathcal{X} \subseteq \mb{R}^n$ is a closed convex set and $g \colon \mb{R}^n \rightarrow \mb{R}\cup\{\infty\}$
is a \emph{gauge} function. Here, we say that $g$ is a gauge function if $g$ is convex, nonnegative, positively homogeneous and satisfies $g(0) = 0$.
Note that GO problems are convex because gauge functions are also convex.
Freund~\cite{F1987} first introduced~\eqref{eq:go}, proposed a dual formulation called the gauge dual
(which differs from the usual Lagrangian dual), and proved some duality results.
He also showed that the class of gauge optimization problems includes the well-known linear programming,
$p$-norm optimization problems with $p \in [1, \infty]$
and convex quadratic optimization problems~\cite{F1987}.

Recently, Friedlander et al.~\cite{FMP2014} considered a specific form of GO problem
in which~$\mathcal{X}$ is described as $\mathcal{X} := \{ x \in \R^n \mid h( b - Ax ) \leq \sigma \}$,
where $h$ is a gauge function, $\sigma$~is a scalar, and $b$, $A$ are, respectively,
a vector and a matrix with appropriate dimensions.
They gave a closed form of its gauge dual.
Afterwards, Friedlander and Mac\^edo~\cite{FM2016} applied this gauge duality
to solve low-rank spectral optimization problems.
Aravkin et al.~\cite{ABD2017} presented some theoretical results for the GO problem.
In particular, they gave optimality conditions, and a way to recover a primal solution from the gauge dual.
In that paper, they also extended their results to a more general convex optimization problem, 
where $g$ and $h$ were not necessarily gauge functions.
In addition, they proposed the perspective duality, which is an extension of the gauge duality.

The gauge optimization problems in these previous works~\cite{ABD2017,F1987,FM2016,FMP2014}
do not involve linear terms in their objective functions.
Therefore, these GO frameworks cannot directly handle linear conic optimization problems. 
More recently, Yamanaka and Yamashita~\cite{YY2017} considered
the following \emph{positively homogeneous optimization} (PHO) problem:
\begin{equation}
\left.
\begin{array}{lrl}
		& \min	 	& c^T x + d^T \Psi(x) \\
 		& \mathrm{s.t.} & A x + B \Psi(x) = b, \\
		& 		& H x + K \Psi(x) \leq p, \\
		&		& x \in \dom\Psi,
\end{array}
\right. 
\tag{\mbox{$\mathrm{P_{PHO}}$}}
\label{eq:PHO}
\end{equation}
where $c\in \mb{R}^n$, $d\in \mb{R}^m$, $b\in \mb{R}^k$, $p\in \mb{R}^\ell$,
$A\in \mb{R}^{k\times n}$, $B\in \mb{R}^{k\times m}$, $H\in \mb{R}^{\ell \times n}$ and $K\in \mb{R}^{\ell \times m}$
are given constant vectors and matrices,
$\Psi \colon \mb{R}^n \rightarrow (\mb{R} \cup \infty)^m$ is defined by
$\Psi(\cdot) := (\psi_1(\cdot), \ldots, \psi_m(\cdot))^T$
where each function $\psi_i \colon \R^{n_i} \rightarrow \R$, $\sum_{i=1}^m n_i = n$ is nonnegative and positively homogeneous,
and $T$ denotes transpose.
Moreover, $\dom\Psi$ denotes the effective domain of $\Psi$, defined by $\dom\Psi := \{ x \in \R^n \mid \psi_i (x_i) < \infty, i = 1,\dots , m \}$
where $x_i \in \R^{n_i}$ is a disjoint subvector of $x$.
Problem~\eqref{eq:PHO} is not necessarily convex, and it includes~\eqref{eq:go} with~$\mathcal{X} = \{ x \in \R^n \mid h( b - Ax ) \leq \sigma \}$
since gauge functions are positively homogeneous.
Note that PHO can handle linear terms in its objective and constraint functions.
Here, we explicitly include $x \in \dom \Psi$ in the constraints of~\eqref{eq:PHO}.
This is because we want to consider more general PHO problems than the ones used in the previous work~\cite{YY2017},
where $\dom\Psi = \R^n$ is assumed.
Then we can adopt an indicator function of some cones as $\psi_i$.
We will later show that the same results as in~\cite{YY2017} can be obtained
even when $\dom\Psi \neq \R^n$.

When $n_i = 1$ and $\psi_i (x_i) = |x_i|$, \eqref{eq:PHO} is reduced to the absolute value programming problem
proposed by Mangasarian~\cite{Ma2007}.
The other examples of PHO problems are 
$p$-order cone optimization problems~\cite{AG2003,XY2000} with $p \in (0, \infty]$,
group Lasso-type problems~\cite{MGB2008,YL2006}, and sum of norms optimization problems~\cite{W2011}.

Yamanaka and Yamashita~\cite{YY2017} proposed a closed-form dual formulation of the PHO,
which they call the positively homogeneous dual, and showed that weak duality holds.
They also investigated the relation between the positively homogeneous dual and the Lagrangian dual
of~\eqref{eq:PHO}, and proved that those problems are equivalent under some conditions.
The result indicates that the Lagrangian dual of a PHO problem can be written in closed form even if it is nonconvex.
Although the PHO problem has the above nice features,
the theoretical analysis is still insufficient.
In particular, the paper~\cite{YY2017} does not discuss strong duality and primal recovery.

In this paper, we mainly study the following gauge optimization problem with possible linear functions:
\begin{equation}
\left.
\begin{array}{lrl}
		& \min	 	& c^T x + d^T \mathcal{G}(x) \\
 		& \mathrm{s.t.} & A x = b, \\
		& 		& H x + K \mathcal{G}(x) \leq p, \\
		&		& x \in \dom \, \mathcal{G},
\end{array}
\right. \tag{P}
\label{eq:MGO}
\end{equation}
where $c, d, b, p, A, H, K$ are the same as in~\eqref{eq:PHO}, 
and $\mathcal{G} \colon \mb{R}^n \rightarrow (\mb{R} \cup \infty)^m$ is defined by
$\mathcal{G}(\cdot) := (g_1(\cdot), \ldots, g_m(\cdot))^T$ with $g_i \colon \R^{n_i} \rightarrow \R$ as a gauge function for all $i$.
Note that there is no nonlinear term in the equality constraints,
and problem~\eqref{eq:MGO} is convex when all elements of $d$ and $K$ are nonnegative.
Problem~\eqref{eq:MGO} includes the convex GO problems considered in~\cite{ABD2017,F1987,FM2016,FMP2014},
and it is possible to explicitly handle linear terms.
In this paper, we call~\eqref{eq:MGO} the gauge optimization problem when it is clear from the context.

In particular, we are interested in theoretical properties of problem~\eqref{eq:MGO} and its dual.
We first define a dual problem of~\eqref{eq:MGO} as in~\cite{YY2017}, and then,
give conditions under which weak and strong dualities hold for problem~\eqref{eq:MGO} and its dual.
Moreover, we present necessary and sufficient optimality conditions for~\eqref{eq:MGO},
that does not use differentials of $g_i$ as in the Karush-Kuhn-Tucker (KKT) conditions.
We further give sufficient conditions under which we can obtain a primal solution
from a KKT point of the dual formulation.
Finally, we show that the theoretical results for problem~\eqref{eq:MGO}
can be extended to general convex optimization problems, by considering the so-called perspective functions.

The paper is organized as follows.
In Section~\ref{sec:pho}, we recall some important properties of~\eqref{eq:PHO} in~\cite{YY2017}.
We show that some of them hold even if $\dom\Psi \neq \R^n$.
Section~\ref{sec:go} presents the dual of  problem~\eqref{eq:MGO},
and gives some relations of~\eqref{eq:MGO} and its dual.
In particular, we show weak and strong duality results, the optimality conditions for the problem,
as well as the recovery of primal solutions by solving the dual problem.
In Section~\ref{sec:gen_convex}, we discuss how to extend the obtained results to general convex optimization problems.
Section~\ref{sec:conclusion} concludes the paper with final remarks and future works.

We use the following notations throughout the paper.
We denote by $\mb{R}_{++}$ the set of positive real numbers.
Let $x \in \mb{R}^n$ be an $n$-dimensional column vector, and $A \in \mb{R}^{n \times m}$ be a matrix
with dimension $n \times m$.
For two vectors $x$ and $y$, we denote the vector $(x^T, y^T)^T$ as $(x, y)^T$ for simplicity.
For a vector $x \in \mb{R}^n$, its $i$-th entry is denoted by~$x_i$.
Moreover, if $I \subseteq \{1, \dots, n\}$, then $x_I$
corresponds to the subvector of $x$ with entries $x_i$, $i \in I$.
The $n$-dimensional vector of ones is given by $e_n$, that is, $e_n := (1,\dots,1)^T \in \mb{R}^{n}$.
The identity matrix with dimension $n$ is $E_n \in \mb{R}^{n \times n}$.
For a matrix $A$, we write $A \succeq 0$ to denote $A$ is symmetric and positive semidefinite.
The notation $\# J$ denotes the number of elements of a set $J$.
We also denote by $\| \cdot \|$ the usual norm.
For a function $f$ and vectors $x$ and $y$, we denote the subdifferential of $f(x, y)$
with respect to $x$ as $\partial_x f(x, y)$.
The effective domain of a function $f$ is given by dom$f$.
The convex hull of a set $S$ is denoted by co$S$.
Finally, $\delta_S \colon \mb{R}^{n} \rightarrow \mb{R} \cup \{\infty \}$ is an indicator function of
a set $S \subseteq \mb R^n$ defined by
\[
\delta_S (x) :=
\left\{ \begin{array}{ll}
0		& \mathrm{if} \: x \in S, \\
\infty		& \mathrm{otherwise}.
\end{array} \right.
\]


\section{Positively homogeneous optimization problems and their duality}
\label{sec:pho}

In this section, we recall positively homogeneous optimization problems and their properties in~\cite{YY2017}.
The positively homogeneous and vector positively homogeneous functions are defined respectively as follows.

\begin{defn}(Positively homogeneous functions)\label{def:phf}
A function $\psi \colon \mb{R}^n \rightarrow \mb{R}\cup\{\infty\}$ is
positively homogeneous if 
$\psi( \lambda x ) = \lambda \psi(x)$ for all $x \in \mb{R}^n$ and $\lambda \in \mb{R}_{++}$.
\end{defn}

\begin{defn}(Vector positively homogeneous functions)\label{def:vphf}
A mapping $\Psi \colon \mb{R}^n \rightarrow (\mb{R} \cup \infty)^m$ is a vector positively homogeneous function
if it is defined as
\[
\Psi(x):=
\left[
\begin{array}{c}
\psi_1(x_{I_1})  \\
\vdots  \\
\psi_m(x_{I_m}) 
\end{array}
\right]
\label{asm:psi2_eq}
\]
with positively homogeneous functions $\psi_i \colon \mb{R}^{n_i} \rightarrow \mb{R}\cup\{\infty\}$,
$i= 1, \ldots ,m$, where $n = n_1 + \cdots + n_m$,
$I_i \subseteq \{1, \ldots, n \}$ is a set of indices satisfying
$I_i \cap I_j = \emptyset$ for all $i \neq j$, and $\#I_i = n_i$.
\end{defn}

The polar positively homogeneous function
associated to a positively homogeneous function~$\psi$ and similarly
the polar vector positively homogeneous function
associated to a vector positively homogeneous function~$\Psi$
are defined as follows.
Note that the paper~\cite{YY2017} calls such polar positively homogeneous functions dual functions.

\begin{defn}(Polar positively homogeneous functions)\label{def:phf_polar}
Let $\psi \colon \mb{R}^n \rightarrow \mb{R}\cup\{\infty\}$ be a positively homogeneous function.
Then, $\psi^\circ \colon \mb{R}^n \rightarrow \mb{R}\cup\{\infty\}$ defined by
\[
\psi^\circ(y) := \sup \{ x^T y \: | \: \psi(x) \leq 1 \}
\]
 is called the polar positively homogeneous function of $\psi$.
\end{defn}

Note that a polar positively homogeneous function is positively homogeneous and convex.
Moreover, when $\psi$ is a norm,  $\psi^\circ$ is the dual norm of $\psi$.

\begin{defn}(Polar vector positively homogeneous functions)\label{def:vphf_polar}
Let $\Psi \colon \mb{R}^n \rightarrow (\mb{R} \cup \infty)^m$ be a vector positively homogeneous function.
A function $\Psi^\circ \colon \mb{R}^n \rightarrow (\mb{R} \cup \infty)^m$ is
the polar vector positively homogeneous function  associated to $\Psi$ if $\Psi^\circ$ is given as
\[
\Psi^\circ(y)=
\left[
\begin{array}{c}
\psi_1^\circ(y_{I_1})  \\
\vdots  \\
\psi_m^\circ(y_{I_m}) 
\end{array}
\right]
\]
with the polar $\psi_i^\circ \colon \mb{R}^{n_i} \rightarrow \mb{R}\cup \{\infty\}$ of positively homogeneous function $\psi_i$,
$i=1,\ldots m$.
\end{defn}

Yamanaka and Yamashita~\cite{YY2017} further assumed
two conditions on positively homogeneous functions for the duality of PHO problems.
The first one is the nonnegativity of positively homogeneous functions.
The second one is that each component $\psi_i$ of $\Psi$ vanishes only at zero and $\dom \Psi = \R^n$.
For example, a usual norm satisfies both conditions, but neither
an indicator function for a cone nor the function $\psi_i (x_{I_i}) = \max\{0, x_{I_i}\}$
satisfies the second condition.
Therefore, the second one is rather restrictive.
Here, we suppose the following weaker assumptions.

\begin{asm}\label{asm:pos_nn}
Each positively homogeneous function $\psi_i$ in $\Psi$ is nonnegative, that is,
$\psi_i(x_{I_i}) \geq 0$ for all $x_{I_i} \in \mb{R}^{n_i}$.
\end{asm}

\begin{asm}\label{asm:pos_zero}
For each $i$, either of the following conditions holds:
\begin{itemize}
	\item[(a)] $d_i \geq 0$, $B_{ji} = 0$ and $K_{ji} \geq 0$ for all $j$,
	\item[(b)] $\dom\, \psi_i = \R^{n_i}$ and there exists $\hat x_{I_i}$ such that $\psi_i (\hat x_{I_i}) \neq 0$.
\end{itemize}
\end{asm}


Note that if problem~\eqref{eq:PHO} satisfies the first condition~(a) of Assumption~\ref{asm:pos_zero}
for all $i$ and all $\psi_i$ are gauge functions, then it becomes a convex gauge optimization problem~\eqref{eq:MGO}.
Since $\dom \Psi \neq \R^n$, we have to show the following lemma that corresponds to ~\cite[Proposition~2.1]{YY2017}.

\begin{lem}\label{lem:ineq}
Let $\Psi$ and $\Psi^\circ$ be a vector positively homogeneous function
and its polar, respectively. Then, we have
\[
\Psi^\circ(y) \geq 0.
\]
In addition, suppose that Assumption~\ref{asm:pos_nn} holds. Then,
\[
\Psi(x)^T \Psi^\circ(y) \geq x^T y
\]
holds for all $x \in \dom\Psi$ and $y \in \dom\Psi^\circ$.
\end{lem}

\begin{proof}
Since the first inequality has been shown in~\cite[Proposition 2.1]{YY2017}
by using Definitions~\ref{def:phf}, \ref{def:phf_polar} and \ref{def:vphf_polar},
we prove only the second inequality.
Clearly, it is enough to show that $\psi_i (x_{I_i}) \psi^\circ_i(y_{I_i}) \geq x_{I_i}^T y_{I_i}$.
For simplicity, we denote $\psi_i$ and $x_{I_i}$
as $\psi$ and $x$, respectively.

If $\psi(x) = 0$, then we can show that $x^T y \leq 0$ for all $y \in \dom\, \psi^\circ$ as follows.
Suppose to the contrary that there exists $y \in \dom\, \psi^\circ$ such that $x^T y > 0$,
and hence $t x^T y \rightarrow \infty$ as $t \rightarrow \infty$.
Moreover, since $\psi(tx) = t \psi(x) = 0$ for all $t>0$,
we have $\psi^\circ (y) \geq \sup \{ tx^T y ~|~ \psi(tx) \leq 1 \} = \infty$,
which contradicts the fact that $y \in \dom\, \psi^\circ$.
Consequently, we obtain $\psi(x) \psi^\circ(y) = 0 \geq x^T y$.

Next we consider the case where $\psi(x) > 0$.
Note that $x \in \dom \Psi$, and hence $\psi (x) < \infty$.
Let $z = x / \psi(x)$.
Since $\psi$ is positively homogeneous, we obtain
\[
\psi(z) = \psi \left( \dfrac{x}{\psi(x)} \right) = \dfrac{1}{\psi(x)}\psi(x) = 1.
\]
Therefore, we have
\begin{eqnarray*}
\psi^\circ (y) = \sup \{ \xi^T y ~|~ \psi(\xi) \leq 1 \} \geq z^T y = \dfrac{1}{\psi(x)} x^T y,
\end{eqnarray*}
which shows the second inequality.
\end{proof}

Yamanaka and Yamashita~\cite{YY2017} proposed the following dual of~\eqref{eq:PHO}:
\begin{equation}
\left.
\begin{array}{lrl}
		& \max	 	& b^T u - p^T v \\
 		& \mathrm{s.t.} & \Psi^\circ (A^T u - H^T v - c) + B^T u - K^T v \leq d, \\
		& 		& v \geq 0,
\end{array}
\right.
\tag{\mbox{$\mathrm{D_{PHO}}$}}
\label{eq:PHO_dual}
\end{equation}
where $(u, v) \in \mb{R}^k \times \mb{R}^\ell$.
Note that if $(u, v)$ is feasible for~\eqref{eq:PHO_dual}, then $A^T u - H^T v - c \in \dom \Psi^\circ$.
For problems \eqref{eq:PHO} and \eqref{eq:PHO_dual}, the following weak duality holds.

\begin{thm}\label{prop:wd}(Weak duality)
Suppose that Assumption~\ref{asm:pos_nn} holds.
Let $x \in \mb{R}^n$ and $(u, v) \in \mb{R}^k \times \mb{R}^\ell$ be feasible solutions
of \eqref{eq:PHO} and \eqref{eq:PHO_dual}, respectively.
Then, the following inequality holds:
\[
c^T x + d^T \Psi(x) \geq b^T u - p^T v.
\]
\end{thm}

\begin{proof}
Using Lemma~\ref{lem:ineq}, we can show weak duality as in the proof of~\cite[Theorem~3.1]{YY2017}.
\end{proof}

Next we show that the optimal values and solutions of problems \eqref{eq:PHO_dual} and
the Lagrangian dual of \eqref{eq:PHO} are the same.
Recall that the Lagrangian dual of \eqref{eq:PHO} is written~as
\begin{equation}
\sup_{\substack{u \\ v \geq 0}} \omega (u, v),
\tag{\mbox{$\mathrm{D_{PHO}^{\mathcal L}}$}}
\label{eq:PHO_Ldual}
\end{equation}
where $\omega \colon \R^k \times \R^\ell \rightarrow \R$ is defined by
\[
\omega(u, v) := \inf_{x \in \dom \Psi} \mathcal{L} (x, u, v)
\]
with the Lagrangian function $\mathcal{L}\colon \R^n \times \R^k \times \R^\ell \rightarrow \R$ of \eqref{eq:PHO}
given by
\[
\mathcal{L}(x, u, v) := c^T x + d^T \Psi(x) + u^T (b - Ax - B\Psi(x)) + v^T(p - Hx - K\Psi(x)).
\]
Note that we explicitly require $x \in \dom \Psi$ in the Lagrangian dual problem~\eqref{eq:PHO_Ldual}.
We prove the following key lemma for the equivalence between~\eqref{eq:PHO_dual} and \eqref{eq:PHO_Ldual}.
Note that it is an extension of~\cite[Lemma~4.1]{YY2017} to the case where $\dom\, \psi_i \neq \R^{n_i}$.

\begin{lem}\label{lem:lag_inf}
Let $\psi_i^\circ$ be the polar positively homogeneous functions of $\psi_i$ for $i=1,\dots, m$.
Suppose that Assumptions~\ref{asm:pos_nn} and \ref{asm:pos_zero} hold.
Assume also that $(\bar u, \bar v)$ with $\bar v \geq 0$
is not a feasible solution of problem~\eqref{eq:PHO_dual}.
Then $\omega(\bar u, \bar v) = -\infty$.
\end{lem}

\begin{proof}
Suppose that $(\bar u, \bar v)$ with $\bar v \geq 0$ is not a feasible solution of \eqref{eq:PHO_dual}.
Then, there exists an index $j$ such that
\begin{eqnarray}\label{eq:infeasible_ineq}
\psi^\circ_j (\alpha_{I_j}) > \beta_j,
\end{eqnarray}
where $\alpha := A^T \bar u - H^T \bar v - c \in \R^n$, and $\beta := d - B^T \bar u + K^T \bar v \in \R^m$.
Let $\bar x := (0, \dots, 0,$ $\bar x_{I_j},$ $0, \dots, 0)$.
Then we have $\Psi(\bar x) = (0, \dots, 0, \psi_j(\bar x_{I_j}), 0, \dots, 0)$ and
\begin{eqnarray}\label{eq:lem_Lag}
\mathcal{L}(\bar x, \bar u, \bar v) = - \alpha_j^T \bar x_{I_j} + \beta_j \psi_j (\bar x_{I_j}) + b^T \bar u + p^T \bar v.
\end{eqnarray}
Now we consider three cases: $\psi^\circ_j (\alpha_{I_j}) \in (0, \infty)$,
$\psi^\circ_j (\alpha_{I_j}) = \infty$, and $\psi^\circ_j (\alpha_{I_j}) = 0$.

First we study the case where $\psi^\circ_j (\alpha_{I_j}) \in (0, \infty)$.
Recall that $\psi^\circ_j (\alpha_{I_j})$ is defined as
\begin{eqnarray}\label{eq:lem_Lag2}
\psi^\circ_j (\alpha_{I_j}) = \sup_{x_{I_j}} \{ x_{I_j}^T \alpha_{I_j} ~|~ \psi_j (x_{I_j}) \leq 1 \}.
\end{eqnarray}
Therefore, for all $\varepsilon > 0$, there exists $\bar x_{I_j}(\varepsilon)$ such that
\begin{eqnarray}\label{eq:lem_Lag2-1}
\psi^\circ_j (\alpha_{I_j}) - \varepsilon \leq \alpha_{I_j}^T \bar x_{I_j}(\varepsilon), \quad \psi_j (\bar x_{I_j}(\varepsilon)) \leq 1.
\end{eqnarray}
Let $\bar \varepsilon$ be a scalar such that
$\bar \varepsilon := \min \{ \psi_j^\circ (\alpha_{I_j}) - \beta_j, \psi_j^\circ (\alpha_{I_j}) \} / 2 > 0$. 
Then $\psi^\circ_j (\alpha_{I_j}) > \bar \varepsilon > 0$.
Moreover, we show that there exists $\bar x_{I_j}$ such that
\begin{eqnarray}\label{eq:lem_Lag3}
\psi^\circ_j (\alpha_{I_j}) - \bar \varepsilon \leq \alpha_{I_j}^T \bar x_{I_j}, \quad \psi_j (\bar x_{I_j}) = 1.
\end{eqnarray}
Since $\psi^\circ_j (\alpha_{I_j}) > \bar \varepsilon$, the inequality~\eqref{eq:lem_Lag2-1} implies $\alpha_{I_j}^T \bar x_{I_j}(\bar \varepsilon) > 0$,
and hence $\bar x_{I_j}(\bar \varepsilon) \neq 0$.
If $\psi_j (\bar x_{I_j}(\bar \varepsilon)) \neq 0$,
then we set $\bar x_{I_j} = \bar x_{I_j}(\bar \varepsilon) / \psi_j (\bar x_{I_j}(\bar \varepsilon))$.
This vector $\bar x_{I_j}$ satisfies conditions~\eqref{eq:lem_Lag3} as shown below.
\[
\psi^\circ_j (\alpha_{I_j}) - \bar \varepsilon \leq \alpha_{I_j}^T \bar x_{I_j}(\bar \varepsilon) \leq \alpha_{I_j}^T \dfrac{\bar x_{I_j}(\bar \varepsilon)}{\psi_j (\bar x_{I_j}(\bar \varepsilon))} = \alpha_{I_j}^T \bar x_{I_j},
\]
\[
\psi_j (\bar x_{I_j}) = \psi_j \left( \dfrac{\bar x_{I_j}(\bar \varepsilon)}{\psi_j (\bar x_{I_j}(\bar \varepsilon))} \right) = \dfrac{1}{\psi_j (\bar x_{I_j}(\bar \varepsilon))} \psi_j (\bar x_{I_j}(\bar \varepsilon)) = 1,
\]
where the second inequality holds from Assumption~\ref{asm:pos_nn} and~\eqref{eq:lem_Lag2-1}.
If $\psi_j (\bar x_{I_j}(\bar \varepsilon)) = 0$, then $\psi_j (t \bar x_{I_j}(\bar \varepsilon)) = t \psi_j (\bar x_{I_j}(\bar \varepsilon)) = 0$
for all $t > 0$ because $\psi_j$ is positively homogeneous. From~\eqref{eq:lem_Lag2}, we have
$\psi^\circ_j (\alpha_{I_j}) \geq t \bar x_{I_j}(\bar \varepsilon)^T \alpha_{I_j}$.
Since $\alpha_{I_j}^T \bar x_{I_j}(\bar \varepsilon) > 0$,
we obtain $\psi^\circ_j (\alpha_{I_j}) \rightarrow \infty$ as $t \rightarrow \infty$,
which is a contradiction.
Therefore, there exists $\bar x_{I_j}$ such that \eqref{eq:lem_Lag3} holds.

We now denote $\bar t = (0, \dots, 0, t \bar x_{I_j}, 0, \dots, 0)$ for $t > 0$.
Then, we have from~\eqref{eq:lem_Lag}
\begin{eqnarray*}
\mathcal{L}(\bar t, \bar u, \bar v) &=& - t \alpha_{I_j}^T \bar x_{I_j} + \beta_j \psi_j (t \bar x_{I_j}) + b^T \bar u + p^T \bar v \\
					&\leq& - t ( \psi^\circ_j (\alpha_{I_j}) - \bar \varepsilon - \beta_j \psi_j (\bar x_{I_j}) ) + b^T \bar u + p^T \bar v \\
					&=& - t ( \psi^\circ_j (\alpha_{I_j}) - \bar \varepsilon - \beta_j ) + b^T \bar u + p^T \bar v,
\end{eqnarray*}
where the second inequality and the third equality hold from~\eqref{eq:lem_Lag3}.
Since $\bar \varepsilon \leq (\psi_j^\circ (\alpha_{I_j}) - \beta_j)/2$, we obtain
\begin{eqnarray*}
\mathcal{L}(\bar t, \bar u, \bar v)	&\leq & - t ( \psi^\circ_j (\alpha_{I_j}) - \bar \varepsilon - \beta_j ) + b^T \bar u + p^T \bar v\\
							&\leq& - t \left( \dfrac{\psi^\circ_j (\alpha_{I_j}) - \beta_j}{2} \right) + b^T \bar u + p^T \bar v,
\end{eqnarray*}
which concludes $\lim_{t \rightarrow \infty} \mathcal{L}(\bar t, \bar u, \bar v) = - \infty$.

Next we consider the case where $\psi^\circ_j (\alpha_{I_j}) = \infty$.
From~\eqref{eq:lem_Lag2}, there exists a sequence $\{ \bar x_{I_j}^k \} \subset \dom\, \psi_j$
such that $\psi_j ( \bar x_{I_j}^k ) \leq 1$ and $(\bar x_{I_j}^k)^T \alpha_{I_j} \rightarrow \infty$ as $k \rightarrow \infty$.
Let $\bar x^k = (0, \dots, 0, \bar x_{I_j}^k, 0, \dots, 0)$. Then, it follows from~\eqref{eq:lem_Lag} that
\[
\mathcal{L}(\bar x^k, \bar u, \bar v) = - \alpha_j^T \bar x_{I_j}^k + \beta_j \psi_j (\bar x_{I_j}^k) + b^T \bar u + p^T \bar v,
\]
and hence $\lim_{k \rightarrow \infty} \mathcal{L}(\bar x^k, \bar u, \bar v) = - \infty$.

We finally study the case where $\psi^\circ_j (\alpha_{I_j}) = 0$.
Note that $0 > \beta_j$ from~\eqref{eq:infeasible_ineq}.
When the first condition~(a) of Assumption~\ref{asm:pos_zero} holds,
it then follows from $\bar v \geq 0$ that $\beta_j = d_j - (B^T \bar u)_j + (K^T \bar v)_j \geq 0$, which is a contradiction.
Now, suppose that the second condition~(b) of Assumption~\ref{asm:pos_zero} holds.
If $\alpha_{I_j} \neq 0$, then there exists $\bar \varepsilon > 0$ such that
$1 \geq \psi_j (\bar \varepsilon \alpha_{I_j}) = \bar \varepsilon \psi_j (\alpha_{I_j})$.
Therefore we have
\[
\psi_j^\circ (\alpha_{I_j}) = \sup_{x_{I_j}} \{ x_{I_j}^T \alpha_{I_j} ~|~ \psi_j (x_{I_j}) \leq 1 \} \geq \bar \varepsilon \alpha_{I_j}^T \alpha_{I_j} > 0,
\]
which is a contradiction.
Now we consider the case where $\alpha_{I_j} = 0$.
From Assumption~\ref{asm:pos_zero}~(b), there exists $\hat x_{I_j}$ such that $\psi_j (\hat x_{I_j}) > 0$.
Let $\hat x(t) = (0, \dots, 0, t \hat x_{I_j}, 0, \dots, 0)$ with $t > 0$.
Then, it follows from~\eqref{eq:lem_Lag} that
\begin{eqnarray*}
\mathcal{L}(\hat x(t), \bar u, \bar v) &=& - \alpha_j^T \hat x_{I_j}(t) + \beta_j \psi_j (t \hat x_{I_j}) + b^T \bar u + p^T \bar v \\
							&=& t \beta_j \psi_j (\hat x_{I_j}) + b^T \bar u + p^T \bar v,
\end{eqnarray*}
and we conclude that $\lim_{t \rightarrow \infty} \mathcal{L}(\hat x(t), \bar u, \bar v) = - \infty$.

Consequently, $\omega(\bar u, \bar v)$ is unbounded from below.
\end{proof}

The next theorem shows that problems~\eqref{eq:PHO_dual} and \eqref{eq:PHO_Ldual}
are equivalent, which means that their optimal values and solutions of those problems are the same\footnote{A reviewer of this manuscript pointed out another
proof of the equivalence. We show the proof in Appendix~A.}.

\begin{thm}
\label{thm:sum}
Suppose that the Lagrangian dual problem~\eqref{eq:PHO_Ldual} has a feasible solution.
Suppose also that Assumptions~\ref{asm:pos_nn} and \ref{asm:pos_zero} hold.
Then, the optimal value and optimal solutions of problem~\eqref{eq:PHO_dual} are the same as those of~\eqref{eq:PHO_Ldual}.
\end{thm}

\begin{proof}
The result can be proved by using Lemma~\ref{lem:lag_inf}
as in the proof of \cite[Theorem~4.1]{YY2017}.
\end{proof}

The next proposition shows that the positively homogeneous dual of problem~\eqref{eq:PHO_dual}
is similar to~\eqref{eq:PHO}.

\begin{prop}\label{prop:shdual}
Suppose that problem~\eqref{eq:PHO_dual} is feasible.
Then, the positively homogeneous dual of \eqref{eq:PHO_dual} can be written as
\begin{equation}
\left.
\begin{array}{lrl}
		& \min	 	& c^T x + d^T y \\
 		& \mathrm{s.t.} & A x + B y = b, \\
		& 		& H x + K y \leq p, \\
		&		& \Psi^{\circ \circ}(x) \leq y,
\end{array}
\right. \tag{\mbox{$\mathrm{P'_{PHO}}$}}
\end{equation}
where $\Psi^{\circ \circ}$ denotes the polar of $\Psi^\circ$, i.e., $\Psi^{\circ \circ} = (\Psi^\circ)^\circ$.
\end{prop}

\begin{proof}
First, note that problem~\eqref{eq:PHO_dual} can be written as
\[
\left.
\begin{array}{lrl}
		& \min	 	& - b^T u + p^T v \\
 		& \mathrm{s.t.} & \Psi^\circ (w) + B^T u - K^T v \leq d, \\
		&			& w = A^T u - H^T v - c, \\
		& 		& - v \leq 0.
\end{array}
\right.
\]
This problem is further reformulated as
\begin{equation}
\label{eqn:pdwd}
\left.
\begin{array}{lrl}
		& \min	 	& \hat{c}^T \theta \\
 		& \mathrm{s.t.} & \hat{K} \hat{\Psi}^\circ(\theta) + \hat{H} \theta \leq \hat{p}, \\
		&			& \hat{A} \theta = c,
\end{array}
\right.
\end{equation}
where $\theta = (u, v, w)^T \in \mb{R}^{k+\ell+n}$,
$\hat{c} = (-b, p, 0)^T \in \mb{R}^{k+\ell+n}$, $\hat{p} = (d, 0)^T \in \mb{R}^{n+\ell}$,
$\hat{A} = (A^T, -H^T, -E_n) \in \mb{R}^{n \times (k+\ell+n)}$,
\[
\hat{K} = 
\left[
\begin{array}{ccc}
0	&	0	&	E_n  \\
0	&	0	&	0
\end{array}
\right] \in \mb{R}^{(n+\ell)\times(k+\ell+n)},
\hat{H} = 
\left[
\begin{array}{ccc}
B^T	&	-K^T	&	0  \\
0	&	-E_{\ell}	&	0
\end{array}
\right] \in \mb{R}^{(n+\ell)\times(k+\ell+n)},
\]
and $\hat{\Psi}^\circ$ is defined by $\hat{\Psi}^\circ(\theta) := (\|u \|_2, \|v \|_2, \Psi^\circ(w))^T$.
Note that $\|u \|_2$ and $\|v \|_2$ in $\hat{\Psi}^\circ$
are dummy functions, and they do not affect the primal problem.

Moreover, the positively homogeneous dual of (\ref{eqn:pdwd}) can be described as
\[
\left.
\begin{array}{lrl}
		& \max	 	& c^T x - \hat{p}^T y \\
 		& \mathrm{s.t.} & \hat{\Psi}^{\circ \circ}( \hat{A}^T x -\hat{H}^T y - \hat{c} ) - \hat{K}^T y \leq 0, \\
		& 		& y \geq 0.
\end{array}
\right.
\]
Let $y = (y_1, y_2)^T$ with $y_1 \in \mb{R}^{n}$ and $y_2 \in \mb{R}^{\ell}$. Then, the above problem can be rewritten as
\begin{equation}
\label{eqn:ddbd}
\left.
\begin{array}{lrl}
		& \min	 	& - c^T x + d^T y_1 \\
 		& \mathrm{s.t.} & \| A x -B y_1 + b \|_2 \leq 0, \\
		&			& \| -H x + K y_1 + y_2 - p \|_2 \leq 0, \\
		&			& \Psi^{\circ \circ}( - x ) - y_1 \leq 0, \\
		& 		& y \geq 0.
\end{array}
\right.
\end{equation}
The first two inequality constraints are equivalent to
\begin{eqnarray*}
- A x + B y_1 &=& b, \\
- H x + K y_1+y_2 &=& p.
\end{eqnarray*}
Since $y_2\geq 0$ in (\ref{eqn:ddbd}), the second equality is further reduced to $- H x + K y_1 \leq p$.
Consequently, we can reformulate (\ref{eqn:ddbd}) as
\[
\left.
\begin{array}{lrl}
		& \min	 	& - c^T x + d^T y_1 \\
 		& \mathrm{s.t.} & - A x + B y_1 = b , \\
		&			& - H x + K y_1 \leq p, \\
		&			& \Psi^{\circ \circ}( - x ) \leq y_1,
\end{array}
\right.
\]
which is precisely $\mathrm{(P'_{PHO})}$ by denoting $-x$ and $y_1$ as $x$ and $y$, respectively.
\end{proof}


\section{Gauge optimization problems and their duality}
\label{sec:go}

In this section, we discuss the following gauge optimization problem:
\begin{equation}
\left.
\begin{array}{lrl}
		& \min	 	& c^T x + d^T \mathcal{G}(x) \\
 		& \mathrm{s.t.} & A x = b, \\
		& 		& H x + K \mathcal{G}(x) \leq p, \\
		&		& x \in \dom \, \mathcal{G}.
\end{array}
\right. \tag{\mbox{$\mathrm{P}$}}
\end{equation}
We call $\mathcal{G}$ a vector gauge function
defined as $\mathcal{G} := (g_1(\cdot), \ldots, g_m(\cdot))^T$
with $g_i \colon \R^{n_i} \rightarrow \R$ as a gauge function for all $i$.
Since~\eqref{eq:MGO} is a special case of~\eqref{eq:PHO},
the PHO dual of~\eqref{eq:MGO} is written as follows:
\begin{equation}
\left.
\begin{array}{lrl}
		& \max	 	& b^T u - p^T v \\
 		& \mathrm{s.t.} & \mathcal{G}^\circ(A^T u - H^T v - c) - K^T v \leq d, \\
		& 		& v \geq 0,
\end{array}
\right. \tag{\mbox{$\mathrm{D}$}}
\label{eq:MGO_dual}
\end{equation}
where $\mathcal{G}^\circ$ is the polar function associated to $\mathcal{G}$.
Here, problem~\eqref{eq:MGO_dual} is a convex optimization problem
since each component $g^\circ_i$ of $\mathcal{G}^\circ$ is convex.

The next proposition is a corollary of Lemma~\ref{lem:ineq}.
Note that since a gauge function is nonnegative, Assumption~\ref{asm:pos_nn} automatically
holds.

\begin{prop}\label{prop:gauge_ineq}
Let $\mathcal{G}$ and $\mathcal{G}^\circ$ be a vector gauge function
and its polar, respectively.
Then, we have
\begin{eqnarray*}
\mathcal{G}^\circ(y) 	&\geq&  0,\\
\mathcal{G}(x)^T \mathcal{G}^\circ(y) &\geq&  x^T y
\end{eqnarray*}
for any $x \in \dom\, \mathcal{G}$ and $y \in \dom\, \mathcal{G}^\circ$.
\end{prop}

\begin{proof}
The proof follows from Lemma~\ref{lem:ineq}.
\end{proof}

We have the weak duality theorem for problems~\eqref{eq:MGO} and~\eqref{eq:MGO_dual}, and
the equivalence between~\eqref{eq:MGO_dual} and the Lagrangian dual of~\eqref{eq:MGO}
from Proposition~\ref{prop:gauge_ineq} and Theorem~\ref{thm:sum}.
Throughout the paper, we denote the Lagrangian dual of~\eqref{eq:MGO} as $\mathrm{(D_\mathcal{L})}$.

\begin{cor}\label{cor:gauge_wd}(Weak duality)
For problems~\eqref{eq:MGO} and~\eqref{eq:MGO_dual}, the following inequality holds:
\[
c^T x + d^T \mathcal{G}(x) \geq b^T u - p^T v
\]
for all feasible points $x \in \mb{R}^n$ and $(u, v) \in \mb{R}^k \times \mb{R}^\ell$
of~\eqref{eq:MGO} and~\eqref{eq:MGO_dual}, respectively.
\end{cor}

\begin{proof}
The proof directly follows from Proposition~\ref{prop:gauge_ineq}.
\end{proof}

\begin{cor}\label{cor:sum}
Suppose that the Lagrangian dual problem~$\mathrm{(D_\mathcal{L})}$ has a feasible solution.
Suppose also that Assumption~\ref{asm:pos_zero} holds.
Then, the optimal value and solutions of problem~\eqref{eq:MGO_dual}
are the same as~$\mathrm{(D_\mathcal{L})}$.
\end{cor}

\begin{proof}
The proof is a direct consequence of Theorem~\ref{thm:sum}.
\end{proof}

We now discuss the strong duality,
necessary and sufficient optimality conditions, and the primal recovery for problem~\eqref{eq:MGO}.
To this end, we need~\eqref{eq:MGO} to be convex.
Thus, from now on, we suppose the following assumption.

\begin{asm}
\label{asm:gauge_conv}
All elements of $d$ and $K$ of problem~\eqref{eq:MGO} are nonnegative.
\end{asm} 

Note that if Assumption~\ref{asm:gauge_conv} holds,
then Assumption~\ref{asm:pos_zero} holds for \eqref{eq:PHO} with $\Psi$ = $\mathcal{G}$.
Moreover, we assume the following condition on each function $g_i$.

\begin{asm}
\label{asm:gauge_lsc}
Each function $g_i$ of $\mathcal{G}$ is lower semi-continuous on $\R^{n_i}$.
\end{asm}

We now show that the dual of~\eqref{eq:MGO_dual} becomes~\eqref{eq:MGO} under
Assumptions~\ref{asm:gauge_conv} and \ref{asm:gauge_lsc}.

\begin{cor}
Suppose that Assumptions~\ref{asm:gauge_conv} and \ref{asm:gauge_lsc} hold.
Assume also that problem~\eqref{eq:MGO_dual} is feasible.
Then, the dual of~\eqref{eq:MGO_dual} is equivalent to~\eqref{eq:MGO}. 
\end{cor}

\begin{proof}
Since $g_i$ is a gauge function for all $i$ and satisfies Assumption~\ref{asm:gauge_lsc},
we have $\mathcal{G}^{\circ \circ}=\mathcal{G}$ by \cite[Theorem~15.1]{R1972}.
Then, it follows from Proposition~\ref{prop:shdual} that the dual of (D) becomes
\begin{equation}
\left.
\begin{array}{lrl}
		& \min	 	& c^T x + d^T y \\
 		& \mathrm{s.t.} & A x = b, \\
		& 		& H x + K y \leq p, \\
		&		& \mathcal{G}(x) \leq y.
\end{array}
\right. \tag{\mbox{$\mathrm{P'}$}}
\label{eq:MGO_prime}
\end{equation}

We show that the optimal value of~\eqref{eq:MGO} is the same as that of~\eqref{eq:MGO_prime}.
Let $x^*$ be an optimal solution of~\eqref{eq:MGO}. Then, $(\bar{x}, \bar{y}):=(x^*, \mathcal{G}(x^*))$ is
feasible for~\eqref{eq:MGO_prime}, and hence $c^T x^* +d^T \mathcal{G}(x^*) \ge 
c^T \bar{x} +d^T \bar{y}$. This shows that the optimal value of~\eqref{eq:MGO_prime}
is less than or equal to that of~\eqref{eq:MGO}.

Next, let $(\hat x, \hat y)$ be an optimal solution of~\eqref{eq:MGO_prime}.
From Assumptions~\ref{asm:gauge_conv} and the fact that $\mathcal{G}(\hat{x}) \leq \hat{y}$, we have
$c^T\hat{x} +d^T\mathcal{G}(\hat{x})\leq  c^T\hat{x} +d^T\hat{y}$ and
$H \hat x + K \mathcal{G}(\hat{x}) \leq H \hat x + K \hat y \leq p$.
Therefore, $(\hat x, \mathcal{G}(\hat x))$ is also optimal for~\eqref{eq:MGO_prime}.
Moreover, $\hat x$ is a feasible solution of~\eqref{eq:MGO}
and $c^T\hat{x} +d^T\mathcal{G}(\hat{x}) \le c^T\hat{x} +d^T\hat{y}$.
The result indicates that the optimal value of~\eqref{eq:MGO} is less than or equal to that of~\eqref{eq:MGO_prime}.

The above discussion shows that the optimal values of~\eqref{eq:MGO} and~\eqref{eq:MGO_prime} are the same.
Furthermore, if $x^*$ is optimal for~\eqref{eq:MGO}, then $(x^* , \mathcal{G}(x^*))$ is optimal for~\eqref{eq:MGO_prime}.
Conversely, if $(\hat x , \hat y)$ is an optimal solution
of~\eqref{eq:MGO_prime}, then $\hat x$ is optimal for~\eqref{eq:MGO}.
\end{proof}


\subsection{Strong duality}

We now focus on the strong duality between problems~\eqref{eq:MGO} and~\eqref{eq:MGO_dual}.
As seen below,
we require a certain constraint qualification for this purpose.

\begin{thm}\label{thm:gauge_sd}(Strong duality)
Suppose that Assumption~\ref{asm:gauge_conv} holds.
Suppose also that the Slater constraint qualification holds for~\eqref{eq:MGO}.
Then, the strong duality holds for problems~\eqref{eq:MGO} and \eqref{eq:MGO_dual}, i.e.,
if~\eqref{eq:MGO} has an optimal solution $x^*$, then \eqref{eq:MGO_dual} also has an optimal solution $(u^*, v^*)$
and the duality gap between~\eqref{eq:MGO} and \eqref{eq:MGO_dual} is zero, that is, $c^T x^* + d^T \mathcal{G}(x^*) = b^T u^* - p^T v^*$.
\end{thm}

\begin{proof}
Suppose that~\eqref{eq:MGO} has a solution.
Since~\eqref{eq:MGO} is convex from Assumptions~\ref{asm:gauge_conv} and the Slater constraint qualification holds for~\eqref{eq:MGO},
the strong duality holds between problems~\eqref{eq:MGO} and $\mathrm{(D_\mathcal{L})}$.
This means that $\mathrm{(D_\mathcal{L})}$ also has an optimal solution and 
the duality gap between~\eqref{eq:MGO} and $\mathrm{(D_\mathcal{L})}$ is zero.
It then follows from Corollary \ref{cor:sum} that 
the optimal value of~\eqref{eq:MGO_dual} is the same as that of~\eqref{eq:MGO}.
Moreover, since an optimal solution of~$\mathrm{(D_\mathcal{L})}$ is that of~\eqref{eq:MGO_dual},
\eqref{eq:MGO_dual} has an optimal solution.
\end{proof}


\subsection{Optimality conditions}

The most well-known optimality conditions in the optimization literature are Karush-Kuhn-Tucker (KKT) conditions.
These KKT conditions use gradients and/or subgradients of the functions involved in the problem.
We now present alternative optimality conditions that do not require gradient information.

We first give sufficient optimality conditions for problems~\eqref{eq:MGO} and \eqref{eq:MGO_dual}.
Note that we do not assume the Slater constraint qualification and
Assumption~\ref{asm:gauge_conv} here.

\begin{thm}\label{thm:gauge_suf_op}(Sufficient optimality conditions)
Points $x^*$ and $(u^*, v^*)$ are optimal for~\eqref{eq:MGO} and \eqref{eq:MGO_dual}, respectively,
if the following conditions hold:
\begin{enumerate}
	\item[$\mathrm{(i)}$] $H x^* + K \mathcal{G}(x^*) \leq p$, $A x^* = b$, $x^* \in \dom\, \mathcal{G},$ \hspace{41mm} (primal feasibility)
	\item[$\mathrm{(ii)}$] $\mathcal{G}^\circ(A^T u^* - H^T v^* - c) - K^T v^* \leq d$, $v^* \geq 0,$ \hspace{42mm} (dual feasibility)
	\item[$\mathrm{(iii)}$] $\left[ d + K^T v^* - \mathcal{G}^\circ(A^T u^* - H^T v^* - c) \right]_i \: g_i(x^*_{I_i}) = 0$, $i=1, \dots, m,$ \hspace{2mm} (complementarity)
	\item[$\mathrm{(iv)}$] $\left[ p - H x^* - K \mathcal{G}(x^*) \right]_i \: v^*_i  = 0$, $i=1,\dots, m,$ \hspace{39mm} (complementarity)
	\item[$\mathrm{(v)}$] $\mathcal{G}^\circ(A^T u^* - H^T v^* - c)^T \mathcal{G}(x^*) = (A^T u^* - H^T v^* - c)^T x^*.$ \hspace{29mm} (alignment)
\end{enumerate}
\end{thm}

\begin{proof}
From the complementarity conditions (iii) and (iv), we obtain
\begin{eqnarray*}
0 &= & \left[ d + K^T v^* - \mathcal{G}^\circ(A^T u^* - H^T v^* - c)\right]^T \mathcal{G}(x^*) + \left[ p - H x^* - K \mathcal{G}(x^*) \right]^T v^* \\
	& = &	d^T \mathcal{G}(x^*) - \mathcal{G}^\circ(A^T u^* - H^T v^* - c)^T \mathcal{G}(x^*) + p^T v^* - (Hx^*)^T v^*.
\end{eqnarray*}
It then follows from the alignment condition that we have
\begin{eqnarray*}
\lefteqn{d^T \mathcal{G}(x^*) - \mathcal{G}^\circ(A^T u^* - H^T v^* - c)^T \mathcal{G}(x^*) + p^T v^* - (Hx^*)^T v^*}  \\
	&=& d^T \mathcal{G}(x^*) - (A^T u^* - H^T v^* - c)^T x^* + p^T v^* - (Hx^*)^T v^*\\
	&=& c^T x^* + d^T \mathcal{G}(x^*) - b^T u^* + p^T v^*,
\end{eqnarray*}
which indicates that the objective function values of the primal and the dual problems are the same
for the feasible points $x^*$ and $(u^*, v^*)$. From the weak duality theorem, 
$x^*$ and $(u^*, v^*)$ are optimal for~\eqref{eq:MGO} and \eqref{eq:MGO_dual}, respectively.
\end{proof}

Note that condition (v) in Theorem~\ref{thm:gauge_suf_op}, called the alignment condition,
is not standard, and seems to be strange at first glance.
This is actually used in the previous work~\cite{ABD2017} about gauge duality,
which is different from the duality considered here. 
Moreover, as it can be seen below, the  alignment condition is one of the necessary conditions for optimality.

When the Slater constraint qualification for problem~\eqref{eq:MGO} and Assumption~\ref{asm:gauge_conv} hold,
the sufficient optimality conditions in Theorem~\ref{thm:gauge_suf_op} become necessary.

\begin{thm}\label{thm:gauge_nec_op}(Necessary conditions for optimality)
Suppose that Assumption~\ref{asm:gauge_conv} holds.
Suppose also that the Slater constraint qualification holds for~\eqref{eq:MGO}.
Let $x^*$ and $(u^*, v^*)$ be optimal solutions of~\eqref{eq:MGO} and \eqref{eq:MGO_dual}, respectively.
Then conditions (i)--(v) in Theorem~\ref{thm:gauge_suf_op} hold. 
\end{thm}

\begin{proof}
Since $x^*$ and $(u^*, v^*)$ are optimal solutions of~\eqref{eq:MGO} and \eqref{eq:MGO_dual}, respectively,
the feasibility conditions (i) and (ii) clearly hold. Moreover, since strong duality holds
for $x^*$ and $(u^*, v^*)$ under the assumptions, we have
\begin{eqnarray*}
0 &=& c^T x^* + d^T \mathcal{G}(x^*) - b^T u^* + p^T v^* \\
  &=& d^T \mathcal{G}(x^*) - (A^T u^* - H^T v^* - c)^T x^* + p^T v^* - (Hx^*)^T v^* \\
 &\geq & d^T \mathcal{G}(x^*) - \mathcal{G}^\circ(A^T u^* - H^T v^* - c)^T \mathcal{G}(x^*) + p^T v^* - (Hx^*)^T v^* \\
 &= & \left[ d + K^T v^* - \mathcal{G}^\circ(A^T u^* - H^T v^* - c)\right]^T \mathcal{G}(x^*) + \left[ p - H x^* - K \mathcal{G}(x^*) \right]^T v^* \\ 
& \geq & 0,
\end{eqnarray*}
where the second equality follows from the fact that $Ax^*=b$, the third inequality follows from 
Proposition~\ref{prop:gauge_ineq}, and the last inequality follows from (i) and (ii).
Thus, the above inequalities hold with equalities, and hence we obtain
conditions (iii), (iv) and (v).
\end{proof}


\subsection{Primal recovery}

Let us now discuss about the recovery of a primal optimal solution from a KKT point of the dual problem~\eqref{eq:MGO_dual}.
For simplicity, we denote $\Phi(u, v) := \mathcal{G}^\circ(A^T u - H^T v - c)$ and
$\phi_i(u, v) := g_i^\circ(A_{I_i}^T u - H_{I_i}^T v - c_{I_i}), i=1,\ldots, m$.
Then, the KKT conditions of~\eqref{eq:MGO_dual} can be described as 
\begin{eqnarray}
p + V^T \lambda - K \lambda - \mu = 0, & \: & V \in \partial_v \Phi(u^*, v^*), \label{eq:kkt1} \\
- b + U^T \lambda = 0, &  & U \in \partial_u \Phi(u^*, v^*),  \label{eq:kkt2} \\
d - \Phi(u^*, v^*) - K^T v^* \geq 0, \lambda \geq 0, &  & \lambda^T (d - \Phi(u^*, v^*) - K^T v^*) = 0, \label{eq:kkt3} \\
v^* \geq 0, \mu \geq 0, &  & v^{*T} \mu = 0, \label{eq:kkt4}
\end{eqnarray}
where $\lambda \in \mb{R}^m$ and $\mu \in \mb{R}^{\ell}$ are Lagrangian multipliers.
Let  $A_i=A_{I_i}$ and $H_i=H_{I_i}$ for all $i=1, \ldots, m$ in the subsequent discussion. Moreover,
we divide matrices $U$ and $V$  as 
\[
U=
\left(
\begin{array}{c}
U_1, \\
\vdots \\
U_m
\end{array}
\right), \;\;
V=
\left(
\begin{array}{c}
V_1, \\
\vdots \\
V_m
\end{array}
\right),
\]
where  $U_i \in \mb{R}^{1 \times k}$ and $V_i \in \mb{R}^{1 \times \ell}$ for all $i=1, \ldots, m$.

We now give the concrete formulae for the subdifferentials $\partial_v \Phi$ and $\partial_u \Phi$.
First, for given  $u \in \mb{R}^k$ and $v \in \mb{R}^{\ell}$,
let us denote $X_i(u, v)$ as the set of optimal solutions of the following problem:
\begin{equation}
\left.
\begin{array}{lrl}
		& \displaystyle{\sup_{x_{I_i}}}	 	& u^T A_i x_{I_i} - v^T H_i x_{I_i} - c_{I_i}^T x_{I_i} \\
 		& \mathrm{s.t.} & g_i(x_{I_i}) \leq 1.
\end{array}
\right.\tag{\mbox{$\mathrm{P}_i$}}
\label{eq:prob_lem}
\end{equation}

Moreover, we assume the following condition to show key properties of~$X_i (u, v)$.

\begin{asm}\label{asm:gauge_zero}
For all $i$, $g_i$ vanishes only at $0$, that is, $g_i(\bar{x}_{I_i})=0$ if and only if $\bar{x}_{I_i} = 0$.
\end{asm}

\begin{lem}\label{lem:gauge_sol_set}
Suppose that  Assumptions~\ref{asm:gauge_lsc} and \ref{asm:gauge_zero} hold.  
Then, $X_i (u, v)$ is nonempty, convex and compact for all $u \in \mb{R}^k$ and $v \in \mb{R}^{\ell}$.
\end{lem}

\begin{proof}
The feasible region of~\eqref{eq:prob_lem} is nonempty since $g_i$ is a gauge function, and
$x_{I_i} = 0$ is a feasible solution of problem~\eqref{eq:prob_lem}.
In addition, the feasible region is convex and closed because each function $g_i$ is convex and closed
from Assumption~\ref{asm:gauge_lsc}.
Moreover, Assumption~\ref{asm:gauge_zero} implies that the feasible region is bounded.
To see this, let $B_i:=\{ z \in \mb{R}^{n_i}\;|\; \|z\|=1\}$ and $\rho:=\inf_{z\in B_i} g_i(z)$.
Then $\rho>0$ from Assumption~\ref{asm:gauge_zero}.
If $\rho=+\infty$, that is, $\dom\, g_i = \{0\}$,
then $X_i(u, v) = \{0\}$ and this lemma holds.
Now, suppose that $\rho < \infty$.
Then, the feasible region is included in the compact set 
$\bar{B_i}:=\{ z\;|\; \|z\| \leq 1/\rho \}$ since for any $s\not \in \bar{B_i}$ we
have  $\|s\|> 1/\rho$ and
\[
g_i(s)=g_i(\|s\| s/\|s\|)=\|s\| g_i(s/\|s\|)> \frac{1}{\rho} \rho =1, 
\]  
which shows that $s$ is not a feasible solution of~\eqref{eq:prob_lem}.
Consequently, the feasible region of~\eqref{eq:prob_lem} is nonempty, convex and compact.

Since~\eqref{eq:prob_lem} is a convex problem with a nonempty, compact and convex feasible region, 
 the optimal solution set of~\eqref{eq:prob_lem} is nonempty, convex and compact.
\end{proof}

We now describe the concrete formulae for $\partial_v \Phi$ and $\partial_u \Phi$ by using $X_i (u, v)$ as follows.
\begin{lem}\label{lem:gauge_sub_dif}
Suppose that Assumptions~\ref{asm:gauge_lsc} and \ref{asm:gauge_zero} hold for function $\mathcal{G}$.
Then, we have 
\begin{equation}
\label{eqn:phi1}
\phi_i (u, v) = u^T A_i \bar{x}_{I_i} - v^T H_i \bar{x}_{I_i} - c_{I_i}^T \bar{x}_{I_i}\;\; for~all~ \bar{x}_{I_i} \in X_i(u, v),
\end{equation}
\begin{equation}
\label{eqn:phi2}
\partial_u \phi_i (u, v) = \{ \bar{x}_{I_i}^T A_i^T ~|~ \bar{x}_{I_i} \in X_i(u, v) \}
\end{equation}
and
\begin{equation}
\label{eqn:phi3}
\partial_v \phi_i (u, v) = \{ - \bar{x}_{I_i}^T H_i^T ~|~ \bar{x}_{I_i} \in X_i(u, v) \}.
\end{equation}
\end{lem}

\begin{proof}
The first equation directly follows from the definitions of $g_i^\circ$ and $X_i(u, v)$.
Since the set $X_i(u, v)$ is nonempty, convex and compact from Lemma~\ref{lem:gauge_sol_set}, we obtain
\begin{eqnarray*}
\partial_u \phi_i (u, v) &=& \mathrm{co} \{ \bar{x}_{I_i}^T A_i^T ~|~ \bar{x}_{I_i} \in X_i(u, v) \} = \{ \bar{x}_{I_i}^T A_i^T ~|~ \bar{x}_{I_i} \in X_i(u, v) \},\\
\partial_v \phi_i (u, v) &=& \mathrm{co} \{ - \bar{x}_{I_i}^T H_i^T ~|~ \bar{x}_{I_i} \in X_i(u, v) \} = \{ - \bar{x}_{I_i}^T H_i^T ~|~ \bar{x}_{I_i} \in X_i(u, v) \},
\end{eqnarray*}
which are the desired formulae.
\end{proof}

Finally, we present the main result of this subsection,
which shows that it is possible to obtain a primal solution from a
KKT point of problem~\eqref{eq:MGO_dual}.

\begin{thm}(Primal recovery)\label{thm:gauge_rec}
Suppose that Assumptions~\ref{asm:gauge_conv}, \ref{asm:gauge_lsc} and \ref{asm:gauge_zero} hold
for the function~$\mathcal{G}$.
Assume also that  $(u^*, v^*, \lambda, \mu)\in  \mb{R}^k\times \mb{R}^\ell \times \mb{R}^m \times \mb{R}^\ell$, $V\in \partial_v \Phi(u^*, v^*)$ and $U \in \partial_u \Phi(u^*, v^*)$ satisfy the KKT conditions (\ref{eq:kkt1})--(\ref{eq:kkt4})
for the dual problem~\eqref{eq:MGO_dual}.
Then there exist  $\bar{x}_{I_i} \in X_i(u^*, v^*)$ for all $i=1,\ldots, m$ such that 
$U_i = (A_i \bar{x}_{I_i})^T$ and 
$V_i = - (H_i \bar{x}_{I_i})^T$. 
Moreover, suppose that $g_i(\bar{x}_{I_i})=1$ for $i$ such that $\lambda_i\neq 0$.
Let $x^*_{I_i} = \lambda_i  \bar{x}_{I_i}$ for all $i=1, \ldots, m$.
Then, $x^* = (x^*_{I_1}, \ldots, x^*_{I_m})^T$ is an optimal solution of~\eqref{eq:MGO}.
\end{thm}

\begin{proof}
From the definitions of $\Phi$ and $\mathcal{G}^\circ$, we have
\[
\Phi(u^*, v^*) = \mathcal{G}^\circ(A^T u^* - H^T v^* - c) =
\left(
\begin{array}{c}
g^\circ_1 (A_1^T u^* - H_1^T v^* - c_{I_1}) \\
\vdots \\
g^\circ_m (A_m^T u^* - H_m^T v^* - c_{I_m})
\end{array}
\right) =
\left(
\begin{array}{c}
\phi_1 (u^*, v^*) \\
\vdots \\
\phi_m (u^*, v^*)
\end{array}
\right).
\]
Moreover, since
\[
U \in \partial_u \Phi(u^*, v^*) \subseteq
\left(
\begin{array}{c}
\partial_u \phi_1 (u^*, v^*) \\
\vdots \\
\partial_u \phi_m (u^*, v^*)
\end{array}
\right),
\]
we have $U_i \in \partial_u \phi_i(u^*, v^*)$.
In a similar way, we have $V_i \in \partial_v \phi_i(u^*, v^*)$.
It then follows from~(\ref{eqn:phi2}) and (\ref{eqn:phi3}) in Lemma~\ref{lem:gauge_sub_dif} that,
for all $i=1,\ldots, m$, there exist $\bar{x}_{I_i} \in X_i(u^*, v^*)$, such that $U_i = (A_i \bar{x}_{I_i})^T$ and 
$V_i = - (H_i \bar{x}_{I_i})^T$.

Now let  $x^*_{I_i} = \lambda_i \bar{x}_{I_i}$, $i=1, \ldots, m$, and $x^* = (x^*_{I_1}, \ldots, x^*_{I_m})^T$.
We show that $x^*$ and $(u^*,v^*)$ satisfy the sufficient conditions (i)--(v) in Theorem \ref{thm:gauge_suf_op}.
Note that the dual feasibility (ii) clearly holds. Moreover, since the assumption on $g_i(\bar{x}_{I_i})$ implies 
\[
g_i(x^*_{I_i})=g_i\left( \lambda_i \bar{x}_{I_i} \right)= \lambda_i g_i(\bar{x}_{I_i})=\lambda_i,
\]
we obtain
\begin{equation}
\label{eqn:gl}
\mathcal{G}(x^*)=\lambda.
\end{equation}

We first show that the alignment condition (v) holds. 
From~(\ref{eqn:phi1}) in Lemma~\ref{lem:gauge_sub_dif}, we have
\[
g_i^\circ(A_i^T u^* - H_i^T v^* - c_{I_i})=\phi_i(u^*, v^*)=(u^*)^T A_i \bar{x}_{I_i} - (v^*)^T H_i \bar{x}_{I_i} - c_{I_i}^T \bar{x}_{I_i}.
\]
It then follows from~(\ref{eqn:gl}) that
\begin{eqnarray*}
g_i^\circ(A_i^T u^* - H_i^T v^* - c_{I_i})^T g_i(x^*_{I_i}) &=& \lambda_i ((u^*)^T A_i \bar{x}_{I_i} - (v^*)^T H_i \bar{x}_{I_i} - c_{I_i}^T \bar{x}_{I_i}) \\
	&=& (u^*)^T A_i x^*_{I_i} - (v^*)^T H_i x^*_{I_i} - c_{I_i}^T x^*_{I_i} \\
	&=& (A_i^T u^* - H_i^T v^* - c_{I_i})^T x^*_{I_i},
\end{eqnarray*}
which shows that condition (v) holds.

Next we prove the primal feasibility (i). From the definition of $x^*$, we obtain
\[
A x^* = \sum_{i=1}^m \lambda_i A_i \bar{x}_{I_i} = \sum_{i=1}^m \lambda_i U_i^T = U^T \lambda=b,
\]
where the second equality follows from~(\ref{eqn:phi2}) in Lemma~\ref{lem:gauge_sub_dif} and the last equality is due to
the KKT condition~\eqref{eq:kkt2}.  
Moreover, we have from (\ref{eqn:phi3}) in Lemma \ref{lem:gauge_sub_dif} that
\begin{equation}
H x^* = \sum_{i=1}^m \lambda_i H_i \bar{x}_{I_i} = - \sum_{i=1}^m \lambda_i V_i^T = - V^T \lambda. 
\label{eq:rec2} 
\end{equation}
It then follows from (\ref{eqn:gl}) that
\[
H x^* + K \mathcal{G}(x^*) = - V^T \lambda+ K \lambda
=p-\mu \leq p,
\]
where the equality and the inequality follow from the KKT conditions~\eqref{eq:kkt1} and \eqref{eq:kkt4}, respectively.
Consequently, $x^*$ is a feasible solution of~\eqref{eq:MGO}.

Finally, we show that the complementarity conditions (iii) and (iv) hold. 
First we consider condition (iii) as follows.
If $\lambda_i=0$, then $x^*_{I_i}=0$ and $g_i(x^*_{I_i})=0$, and hence (iii) holds.
If $\lambda_i\neq 0$, then $\left[ d + K^T v^* - \mathcal{G}^*(A^T u^* - H^T v^* - c) \right]_i=0$ from the KKT condition (\ref{eq:kkt3}) and the definition of $\Phi$.
Therefore, (iii) also holds.

Next we prove that condition (iv) is satisfied.
If $v^*_i=0$, then (iv) clearly holds. For this reason, we consider the case where $v^*_i\neq 0$.
In such a case, $\mu_i=0$ from the KKT condition (\ref{eq:kkt4}), and hence $\left[ p+V^T\lambda-
K\lambda\right]_i=0$ from the KKT condition (\ref{eq:kkt1}). 
It then follows from (\ref{eqn:gl})  and  (\ref{eq:rec2})  that
\[
0=
\left[p+V^T\lambda-
K\lambda\right]_i =\left[p-Hx^*-
K\lambda\right]_i = \left[p-Hx^*-
 K \mathcal{G}(x^*) \right]_i.
\]
Therefore, the complementarity condition (iv) holds.

From the previous discussion, we conclude that $x^*$ and $(u^*, v^*)$ satisfy all sufficient conditions for optimality, 
and hence $x^*$ is an optimal solution of~\eqref{eq:MGO}. 
\end{proof}

Observe that the assumption that $g_i(\bar x_{I_i})=1$ for all $i$ such that $\lambda_i \neq 0$
seems to be rather restrictive. One sufficient condition
for the assumption is that the effective domain of $g_i$  is $\mb{R}^{n_i}$  and
$A_i^T u^*-H_i^Tv^*-c_{I_i}\neq 0$ for all $i$. Under these conditions, the solution set $X_i(u^*,v^*)$
is included in the boundary of the feasible set of~\eqref{eq:prob_lem}, and thus $g_i(\bar{x}_{I_i})=1$ for all 
$\bar{x}_{I_i} \in X_i(u^*,v^*)$.


\section{Duality for general convex optimization}
\label{sec:gen_convex}

In this section we extend the previous results for gauge optimization to more general convex optimization problems.
To this end, we first decompose general convex function of the problem, which is not necessarily nonnegative,
into a linear and a nonnegative convex functions.
Then, we consider the so-called \emph{perspective}~\cite{ABD2017,ABF2013} for the nonnegative convex function.
The perspective function is a gauge one essentially equivalent to the original nonnegative convex function.
Consequently, we reformulate the general convex function into a sum of linear function and a gauge one.
The reformulation enables us to apply the results in the previous section for a general convex optimization problem.

\subsection{Reformulation of a general convex function into sum of  linear and gauge functions}\label{sec:dec}

Let us first observe that a convex function~$f \colon \R^n \rightarrow \R \cup \{\infty \}$
can be written as a sum of a linear function and a nonnegative convex one.
Let $z \in \dom f$ be a fixed vector, and let $\eta \in \partial f(z)$.
We can write
\begin{eqnarray}\label{eq:reform}
f(x) = f(x) - f(z) - \eta^T ( x-z ) + f(z) + \eta^T(x-z).
\end{eqnarray}
Note that $f(x)-f(z)-\eta^T(x-z)$ is convex and nonnegative with respect to $x$,
because $f$ satisfies the subgradient inequality \cite[p. 214]{R1972}: $f(x) \geq f(z) + \eta^T (x - z)$.
Moreover, the remaining term: $f(z)+\eta^T(x-z)$ is linear with respect to $x$.
Thus, function $f$ can be split into a nonnegative convex function and a linear one.

Next, we reformulate a nonnegative convex function into a gauge function through the so-called \emph{perspective} of a nonnegative convex function.
Recall that for any nonnegative convex function $h \colon \mb{R}^n \rightarrow \mb{R}_+ \cup \{\infty \}$,
its perspective $h^p \colon \mb{R}^{n+1} \rightarrow \mb{R} \cup \{\infty \}$ is described as
\[
h^p (x, \zeta) :=
\left\{ \begin{array}{ll}
\zeta h(\zeta^{-1} x)	& \mathrm{if} \: \zeta > 0, \\
\delta_{\{0\}}(x)				& \mathrm{if} \: \zeta = 0, \\
\infty						& \mathrm{if} \: \zeta < 0,
\end{array} \right.
\]
and its closure can be written by
\begin{eqnarray}\label{eq:pers_def}
h^\pi (x, \zeta) :=
\left\{ \begin{array}{ll}
\zeta h(\zeta^{-1} x)	& \mathrm{if} \: \zeta > 0, \\
h^{\infty}(x)				& \mathrm{if} \: \zeta = 0, \\
\infty						& \mathrm{if} \: \zeta < 0,
\end{array} \right.
\end{eqnarray}
where $h^{\infty}$ is the \emph{recession function} of $h$~\cite[p. 66]{R1972}.
Note that if $h$ is a proper convex function,
then $h^\pi$ is a positively homogeneous proper convex function~\cite[Theorem 8.5]{R1972}.
In addition, $h^\pi(0, 0) = 0$ by definition, and hence $h^\pi$ becomes gauge.
Therefore, $h$ is represented as the gauge function $h^\pi (x, \zeta)$ with $\zeta = 1$.
Consequently, $f$ can be described as a sum of the linear function $f(z) + \eta^T (x-z)$ and a gauge function $h^\pi(x, 1)$,
where $h(x) = f(x) - f(z) - \eta^T (x- z)$.
We present an example of perspective and its polar.

\begin{exm}\label{exm:quad}
Let $f \colon \R^n \rightarrow \R$ be defined as $f(x) := \frac{1}{2} x^T A x$, where $A$ is an $n \times n$
symmetric positive definite matrix.
Then, the perspective and its polar of the quadratic function $f$ are described as follows:
\[
\begin{array}{rcl}
f^\pi(x, \zeta)		& = &	\left\{ \begin{array}{cl}
						\dfrac{1}{2 \zeta}x^T A x	& \mathrm{if} \: \zeta > 0, \\
						\delta_{\{0\}}(x)			& \mathrm{if} \: \zeta = 0, \\
						\infty						& \mathrm{otherwise},
						\end{array} \right. \vspace{1mm} \\
f^\natural(y, \eta)	& = &	\left\{ \begin{array}{cl}
						- \dfrac{1}{2 \eta}y^T A^{-1} y	& \mathrm{if} \: \eta < 0, \\
						\delta_{\{0\}}(y)				& \mathrm{if} \: \eta = 0, \\
						\infty						& \mathrm{otherwise}.
						\end{array} \right.
\end{array}
\]
\end{exm}

\begin{proof}
The perspective $f^\pi$ directly follows from definition \eqref{eq:pers_def}.
Note that $A$ is positive definite, hence $f^\infty = \delta_{\{0\}}$ \cite[p. 68]{R1972}.
The polar of $f^\pi$ is defined by
\begin{eqnarray}\label{ex1_fn}
f^\natural (y, \eta) = \sup_{x, \zeta} \left\{ x^T y + \zeta \eta \mid f^\pi(x, \zeta) \leq 1 \right\}.
\end{eqnarray}

We first consider the case where $\eta > 0$.
Since $f^\pi(0, \zeta) = 0$ for $\zeta \ge 0$,
$f^\natural (y, \eta) \geq \zeta \eta$ for $\zeta \ge 0$.
Then $f^\natural (y, \eta) \rightarrow \infty$ as $\zeta \rightarrow \infty$.
Next suppose that $\eta = 0$ and $y \neq 0$.
Let $x(t) := t y$ with $t > 0$, and let $\zeta(t) := \frac{1}{2} x(t)^T A x(t)$.
Since $A$ is positive definite, $\zeta(t) = \frac{1}{2} x(t)^T A x(t) > 0$.
Then $f^\pi (x(t), \zeta(t)) = 1$ for all $t$.
Consequently $f^\natural (y, 0) \ge x(t)^T y + \zeta(t) \cdot 0 = t \|y \|^2$,
and hence $f^\natural (y, 0) \rightarrow \infty$ as $t \rightarrow \infty$.

Next, we study the case where $y = 0$ and $\eta \le 0$.
If $(y, \eta) = (0,0)$, then $f^\natural (y, \eta) = 0$.
Note that $f^\pi (x, \zeta) \le 1$ implies $\zeta \ge 0$, and $f^\pi (0, 0) \le 1$.
Therefore, when $y = 0$ and $\eta < 0$ we have $f^\natural (y, \eta) = 0$.

Finally, we investigate the case where $y \neq 0$ and $\eta < 0$.
We now set
\begin{eqnarray}\label{ex_1}
x^* = - \dfrac{1}{\eta} A^{-1} y, \quad \zeta^* = \dfrac{1}{2 \eta^2} y^T A^{-1} y,
\end{eqnarray}
and
\begin{eqnarray}\label{ex_kkt1}
\lambda^* = - \eta \dfrac{2(\zeta^*)^2}{(x^*)^T A x^*}.
\end{eqnarray}
Since $x^* \neq 0$ and $\zeta^* > 0$, $\lambda^*$ is well-defined and $\lambda^* > 0$.
Moreover, we have from~\eqref{ex_1}
\[
\dfrac{1}{2}(x^*)^T A x^* = \dfrac{1}{2 \eta^2} y^T A^{-1} y = \zeta^*.
\]
It then follows from~\eqref{ex_kkt1} that
\begin{eqnarray}\label{ex_2}
\eta = - \dfrac{\lambda^*}{\zeta^{*}}.
\end{eqnarray}
Then, equations~\eqref{ex_1} and \eqref{ex_2} give
\begin{eqnarray}\label{ex_kkt2}
- y + \dfrac{\lambda^*}{\zeta^*} A x^* = 0.
\end{eqnarray}
We note that the following conditions also hold:
\begin{eqnarray}
\dfrac{1}{2\zeta^*} x^{*T} A x^* - 1 \leq 0, \lambda^* \geq 0, \label{ex_kkt3} \\
\lambda^* \left(\dfrac{1}{2\zeta^*} x^{*T} A x^* - 1 \right) = 0. \label{ex_kkt4}
\end{eqnarray}
Note also that $f^\pi(x, \zeta) = \frac{1}{2\zeta} x^T A x$.
Conditions~\eqref{ex_kkt1}, \eqref{ex_kkt2}, \eqref{ex_kkt3} and \eqref{ex_kkt4}
are the KKT conditions of the convex optimization problem in the right-hand of~\eqref{ex1_fn}.
Therefore, the point $(x^*, \zeta^*)$ is its global optimal solution.
Consequently, we obtain
\[
f^\natural (y, \eta) = (x^*)^T y + \zeta^* \eta = - \dfrac{1}{2 \eta} y^T A^{-1} y,
\]
which completes the proof.
\end{proof}

We now consider a vector function $F \colon \mb{R}^n \rightarrow (\mb{R} \cup \infty)^m$,
which is defined by
$F(\cdot) := (f_1(\cdot), \ldots, f_m(\cdot))$ with nonnegative convex functions
$f_i \colon \mb{R}^{n_i} \rightarrow \mb{R} \cup \{\infty\}$, $i=1,\dots, m$.
We then define its perspective $F^\pi \colon \mb{R}^{n+m} \rightarrow (\mb{R} \cup \infty)^m$ as
$F^\pi(\cdot) := (f^\pi_1(\cdot), \ldots, f^\pi_m(\cdot))$ with $f^\pi_i \colon \mb{R}^{n_i + 1} \rightarrow \mb{R} \cup \{\infty\}$.
For simplicity, we denote $F^\pi(x, \zeta) = (f^\pi_1(x_1, \zeta_1), \ldots, f^\pi_m(x_m, \zeta_m))$
for any $x \in \mb R^n$ and $\zeta \in \mb R^m$.
We also denote the polar of $F^\pi$ as
$F^\natural (\cdot) := (F^\pi)^\circ (\cdot) = ((f^\pi_1)^\circ(\cdot), \ldots, (f^\pi_m)^\circ(\cdot))$.
Note that $F^\pi (x, e_m) = (f^\pi_1(x_1, 1), \ldots, f^\pi_m(x_m, 1)) = F (x)$ by definition.
We also observe that $F^\pi$ is a vector gauge function
if $f_i$ is an nonnegative proper convex function for all~$i$.


\subsection{Perspective dual problems}

We now consider the following nonconvex optimization problem:
\begin{equation}
\left.
\begin{array}{lrl}
		& \min	 	& c^T x + d^T F(x) \\
 		& \mathrm{s.t.} & A x = b, \\
		& 		& H x + K F(x) \leq p,
\end{array}
\right. \tag{\mbox{$\mathrm{P}_F$}}
\label{eq:PF}
\end{equation}
where $F$ is an nonnegative vector convex function,
that is, each component function $f_i$ is an nonnegative convex function.
By using the perspective function of $F$,
we reformulate~\eqref{eq:PF} into a gauge optimization:
\begin{equation}
\left.
\begin{array}{lrl}
		& \min	 	& \hat c^T z + d^T F^\pi(z) \\
 		& \mathrm{s.t.} & \hat A z = \hat b, \\
		& 		& \hat H z + K F^\pi(z) \leq p,
\end{array}
\right. \tag{\mbox{$\mathrm{P}_\pi$}}
\label{eq:Ppi}
\end{equation}
where $F^\pi \colon \mb{R}^{n+m} \rightarrow \mb{R}^m$ is the perspective of $F$, 
$z = (x_{I_1}, \zeta_1, \dots, x_{I_m}, \zeta_m)^T \in \mb{R}^{n+m}$,
$\hat c = (c_{I_1}, 0, \ldots, c_{I_m} , 0)^T \in \mb{R}^{n+m}$,
$\hat b = (b, 1, \dots, 1)^T \in \mb{R}^{2m}$,
$\hat H = [H_{I_1}, 0, \dots, H_{I_m}, 0] \in \mb{R}^{\ell \times (n+m)}$ and
\[
\hat A =
\left[
\begin{array}{ccccc}
A_{I_1}		& 0	&	\cdots 		& A_{I_m} 	& 0 \\
0 			& 1	&	\cdots		& 0			& 0	 \\
\vdots	&\vdots	& \ddots	&\vdots	&\vdots	\\
0	&0	&\cdots	&0	&1
\end{array}
\right] \in \mb{R}^{2m \times (n+m)},
\]
where $A_{I_i}$ is a submatrix of $A$ with $A_j, j\in I_i$ as its columns.

We obtain the PHO dual of~\eqref{eq:Ppi} as follows:
\begin{equation}
\left.
\begin{array}{lrl}
		& \max	 	& b^T u - p^T v + e_m^T w \\
 		& \mathrm{s.t.} & F^\natural
		\left(
		\begin{array}{c}		
		(A_{I_1})^T u - (H_{I_1})^T v - c_{I_1} \\
		w_1 \\
		\vdots \\
		(A_{I_m})^T u - (H_{I_m})^T v - c_{I_m} \\
		w_m
		\end{array}
		\right) - K^T v \leq d, \\
		& 		& v \geq 0.
\end{array}
\right. \tag{\mbox{$\mathrm{D}_\pi$}}
\label{eq:Dpi}
\end{equation}
We call problem~\eqref{eq:Dpi} as the perspective dual of~\eqref{eq:PF}.

We now consider the following convex quadratic optimization problem as an example of~\eqref{eq:PF}.
\begin{equation}
\left.
\begin{array}{lrl}
		& \min	 	& \dfrac{1}{2} x^T A_0 x + b_0^T x \vspace{1mm} \\
 		& \mathrm{s.t.} & \dfrac{1}{2} x^T A_1 x + b_1^T x \le c_1,
\end{array}
\right. \tag{\mbox{$\mathrm{P_{QP}}$}}
\label{eq:qp}
\end{equation}
where $A_0$ and $A_1$ are symmetric and positive definite matrices.
The problem can be rewritten as
\[
\left.
\begin{array}{lrl}
		& \min	 	& \dfrac{1}{2} x^T A_0 x + b_0^T x \vspace{1mm} \\
 		& \mathrm{s.t.} & \dfrac{1}{2} y^T A_1 y + b_1^T y \le c_1, \\
		&			& x - y = 0.
\end{array}
\right.
\]
Let $z := (x, y)^T$ and $F(z) := (f_0(x), f_1(y))^T = (\frac{1}{2} x^T A_0 x, \frac{1}{2} y^T A_1 y)^T$.
Then the problem is described as follows:
\[
\left.
\begin{array}{lrl}
		& \min	 	& (b_0^T, 0) z + (1, 0) F(z) \vspace{1mm} \\
 		& \mathrm{s.t.} & (0, b_1^T) z + (0, 1) F(z) \le c_1, \\
		&			& (I, -I) z = 0.
\end{array}
\right.
\]
Let $w := (x, \zeta_1, y, \zeta_2) \in \mb R^{2n+2}$ and $F^\pi(w) := (f_0^\pi(x, \zeta_1), f_1^\pi(y, \zeta_2))$.
Then, a gauge optimization~\eqref{eq:Ppi} equivalent to~\eqref{eq:qp} is written as
\begin{equation}
\left.
\begin{array}{lrl}
		& \min	 	& (b_0^T, 0, 0, 0) w + (1, 0) F^\pi(w) \vspace{1mm} \\
 		& \mathrm{s.t.} & (0, 0, b_1^T, 0) w + (0, 1) F^\pi(w) \le c_1, \vspace{1mm} \\
		&			& \left[
\begin{array}{cccc}
I		& 0	&	-I 		&  0 \\
0 		& 1	&	0		&  0	 \\
0		& 0	&	0		& 1
\end{array}
\right]  w = \left[
		\begin{array}{c}		
		0 \\
		1 \\
		1
		\end{array}
		\right].
\end{array}
\right.\tag{\mbox{$\mathrm{P}_\pi^\mathrm{QP}$}}
\label{eq:qp_pi}
\end{equation}
Let $F^\natural := (f_0^\natural, f_1^\natural)$ be the polar of $F^\pi$.
Then the PHO dual of~\eqref{eq:qp_pi} is given as
\[
\left.
\begin{array}{lrl}
		& \max	 	& (0, 1, 1) u - c_1 v \\
 		& \mathrm{s.t.} & F^\natural
		\left(
		\left[
		\begin{array}{ccc}		
		I & 0 & 0 \\
		0 & 1 & 0 \\
		-I & 0 & 0 \\
		0 & 0 & 1
		\end{array}
		\right]u - \left[
		\begin{array}{c}		
		0 \\
		0 \\
		b_1 \\
		0
		\end{array}
		\right] v - \left[
		\begin{array}{c}		
		b_0 \\
		0 \\
		0 \\
		0
		\end{array}
		\right]
		 \right) - \left[
		\begin{array}{c}		
		0 \\
		1
		\end{array}
		\right] v \leq \left[
		\begin{array}{c}		
		1 \\
		0
		\end{array}
		\right], \\
		& 		& v \geq 0.
\end{array}
\right.
\]
Let $u = (u_1, u_2, u_3)^T \in \R^n \times \R \times \R$.
Then the dual problem can be further rewritten as
\begin{equation}
\left.
\begin{array}{lrl}
		& \max	 	& u_2 + u_3 - c_1v \\
 		& \mathrm{s.t.} & f_0^\natural (u_1 - b_0, u_2) \le 1, \\
		&			& f_1^\natural (-u_1 - b_1 v, u_3) \le v, \\
		&			& v \ge 0.
\end{array}
\right.\tag{\mbox{$\mathrm{D}_\pi^\mathrm{QP}$}}
\label{eq:qp_pi_dual}
\end{equation}
Recall that the functions $f_0^\natural$ and $f_1^\natural$ are described as in Example~\ref{exm:quad}.
It is easy to see that $(u, v)$ with $u_2 > 0$ or $u_3 > 0$ is not feasible for~\eqref{eq:qp_pi_dual}.

The following lemma indicates the first two constraints in~$\eqref{eq:qp_pi_dual}$
can be represented as semidefinite constraints.

\begin{lem}\label{lem:sdp}
Let $f(x)=\frac{1}{2} x^T A x$,
where $A$ is an $n \times n$ symmetric and positive definite matrix.
Then,
\begin{eqnarray}\label{eq:sdp_ineq}
f^\natural(y,\eta) \le \gamma, \quad \gamma \ge 0
\end{eqnarray}
if and only if 
\begin{eqnarray}\label{eq:sdp}
\left[
\begin{array}{cc}
A \gamma 	&  y \\
y^T		&  -2 \eta
\end{array}
\right] \succeq 0.
\end{eqnarray}
\end{lem}

\begin{proof}
First we suppose that \eqref{eq:sdp_ineq} holds.
The inequality $f^\natural(y,\eta) \le \gamma$ implies $(y, \eta) = (0, 0)$ or $\eta < 0$
from the definition of $f^\natural$ in Example~\ref{exm:quad}.
If $(y, \eta) = (0, 0)$,
then \eqref{eq:sdp} holds since $A$ is positive definite and $\gamma \geq 0$.
If $\eta < 0$, then \eqref{eq:sdp_ineq} can be written as
\[
-\dfrac{1}{2 \eta} y^T A^{-1} y \le \gamma, \quad \gamma \ge 0.
\]
If $\gamma = 0$, then we have $y=0$, and hence \eqref{eq:sdp} holds.
When $\gamma > 0$, we obtain
\begin{eqnarray}\label{eq:sdp_ineq_2}
-2 \eta - \dfrac{1}{\gamma} y^T A^{-1} y \ge 0, \quad \gamma \ge 0,
\end{eqnarray}
which results in \eqref{eq:sdp} by using the Schur complement~\cite{BV2004}.

Next we assume that~\eqref{eq:sdp} holds.
Then, we have $\eta \le 0$ and $\gamma \ge 0$.
If $\eta = 0$, then $y = 0$ from \eqref{eq:sdp}.
It then follows from Example~\ref{exm:quad} that
$f^\natural(y,\eta) = \delta_{\{0\}}(0) = 0 \le \gamma$, and hence~\eqref{eq:sdp_ineq} holds.
If $\gamma = 0$, then $y = 0$ once again.
Then we obtain $f^\natural(y,\eta) = 0 = \gamma$,
which indicates~\eqref{eq:sdp_ineq} holds.
If $\eta < 0$ and $\gamma > 0$, then the Schur complement of~\eqref{eq:sdp}
gives~\eqref{eq:sdp_ineq_2},
which results in~\eqref{eq:sdp_ineq}.
\end{proof}

From Lemma~\ref{lem:sdp}, the perspective dual problem~\eqref{eq:qp_pi_dual}
of problem~\eqref{eq:qp} is equivalent to
the following semidefinite programming~\cite{VB1996}:
\[
\left.
\begin{array}{lrl}
		& \max	 	& u_2 + u_3 - c_1v \vspace{1mm} \\ 
		& \mathrm{s.t.} &  \left[
		\begin{array}{cc}		
		A_0 & u_1 - b_0  \\
		(u_1 - b_0)^T & -2 u_2
		\end{array}
		\right] \succeq 0, \vspace{1mm} \\
		&			& \left[
		\begin{array}{cc}		
		A_1 v & u_1 + b_1 v  \\
		(u_1 + b_1 v)^T & -2u_3
		\end{array}
		\right] \succeq 0.
\end{array}
\right.
\]


\section{Conclusion}
\label{sec:conclusion}

In this paper, we considered optimization problems with both gauge functions and linear functions
in their objective and constraint functions.
Using the positively homogeneous framework given in~\cite{YY2017},
we proved that weak and strong duality results hold for such gauge problems. 
We also discussed both necessary and sufficient optimality conditions associated to these problems,
showing that it is possible to obtain a primal solution by solving the dual problem.
We also extended the results for gauge problems to general convex optimization problems.
An important future work is to develop an efficient algorithm by using the theoretical results described here.


\appendix

\section{Appendix}
\label{sec:appendix}

If each function $\psi_i$, $i=1,\ldots, m$ is gauge,
then the equivalence of~\eqref{eq:PHO_dual} and~\eqref{eq:PHO_Ldual}
is proved as follows.

The Lagrangian function of~\eqref{eq:PHO} is written by
\begin{eqnarray*}
{\cal L}(x, u, v) &=& c^T x + b^T \Psi(x) + u^T (b - A x - B \Psi(x)) - v^T (p - H x - K \Psi(x)) \\
    &=& b^T u - p^T v + x^T (H^T v + c - A^T u) + \Psi(x)^T (d - B^T u + K^T v).
\end{eqnarray*}
Then we obtain the dual function as
\begin{eqnarray*}
\omega(u,v) &=& \inf_{x\in \mathrm{dom}\Psi} {\cal L}(x, u, v)\\
    &=& b^T u - p^T v + \inf_{x\in \mathrm{dom}\Psi} \{ x^T (H^T v + c - A^T u) + \Psi(x)^T (d - B^T u + K^T v) \}\\
    &=& b^T u - p^T v + \sup_{x\in \mathrm{dom}\Psi} \{ x^T (A^T u - H^T v - c) - \Psi(x)^T (d - B^T u + K^T v) \}\\
    &=& b^T u - p^T v + \sum_{i=1}^m \sup_{x\in \mathrm{dom}\Psi} \{ x_{I_i}^T (A_{I_i}^T u - H_{I_i}^T v - c_{I_i}) - \psi_i(x_{I_i}) (d_i - B_i^T u + K_i^T v) \}.
\end{eqnarray*}
The third summation term can be rewritten as
\begin{eqnarray*}
\lefteqn{\sum_{i=1}^m \sup_{x\in \mathrm{dom}\Psi} \{ x_{I_i}^T (A_{I_i}^T u - H_{I_i}^T v - c_{I_i}) - \psi_i(x_{I_i}) (d_i - B_i^T u + K_i^T v) \}}  \\
	&=& \sum_{i=1}^m (d_i - B_i^T u + K_i^T v) \sup_{x\in \mathrm{dom}\Psi} \left\{ \dfrac{x_{I_i}^T (A_{I_i}^T u - H_{I_i}^T v - c_{I_i})}{(d_i - B_i^T u + K_i^T v)} - \psi_i(x_{I_i}) \right\}\\
	&=& \sum_{i=1}^m (d_i - B_i^T u + K_i^T v) \sup_{x\in \mathrm{dom}\Psi} \left\{ x_{I_i}^T y(u,v) - \psi_i(x_{I_i}) \right\}\\
	&=& \sum_{i=1}^m (d_i - B_i^T u + K_i^T v) \psi_i^*(y(u,v)),
\end{eqnarray*}
where we denote $y(u,v) = (A_{I_i}^T u - H_{I_i}^T v - c_{I_i}) / (d_i - B_i^T u + K_i^T v)$, and the third equality follows from the definition of conjugate function. From item (iv) of Proposition 2.1 in~\cite{FMP2014}, we have
\[
\sum_{i=1}^m (d_i - B_i^T u + K_i^T v) \psi_i^*(y(u,v)) = \sum_{i=1}^m (d_i - B_i^T u + K_i^T v) \delta_{\Omega_i}(y(u,v)),\\
\]
where $\Omega_i := \{(u,v): \psi_i^\circ(A_{I_i}^T u - H_{I_i}^T v - c_{I_i}) \leq d_i - B_i^T u + K_i^T v \}$. Finally, we obtain the Lagrangian dual problem of~\eqref{eq:PHO} as
\begin{equation*}
\left.
\begin{array}{lrl}
		& \max	 	& b^T u - p^T v \\
 		& \mathrm{s.t.} & \Psi^\circ(A^T u - H^T v - c) + B^T u - K^T v \leq d, \\
		& 		& v \geq 0,
\end{array}
\right.
\end{equation*}
which is the same as problem~\eqref{eq:PHO_dual}.


\bibliographystyle{plain}
\bibliography{reference}

\end{document}